\documentclass[11pt]{amsart}
\usepackage{amssymb,hhline,array,graphicx}

\newcommand{\Aut}{\mathrm{Aut\mkern 2mu}}

\newcommand{\Hol}{\mathsf{Hol}}

\newcommand{\F}{\mathcal {F}}
\newcommand{\A}{\mathcal A}

\newcommand{\Z}{\mathcal Z}

\newcommand{\IZ}{\mathbb Z}

\newtheorem{theorem}{Theorem}[section]
\newtheorem{proposition}[theorem]{Proposition}
\newtheorem{corollary}[theorem]{Corollary}
\newtheorem{problem}[theorem]{Problem}

\title{Automorphism groups of superextensions of groups} 
\author{Taras Banakh and Volodymyr Gavrylkiv}
\address{T.Banakh: Ivan Franko National University of Lviv (Ukraine), and
Institute of Mathematics, Jan Kochanowski University in
Kielce (Poland)}
\email{t.o.banakh@gmail.com}
\address{V.Gavrylkiv: Faculty of Mathematics and Computer Science,
Vasyl Stefanyk Precarpathian National University,
Shevchenko Street 57, Ivano-Frankivsk, 76025, Ukraine}
\email{vgavrylkiv@gmail.com}

\begin{document}

\begin{abstract}
A family $\mathcal L$ of
subsets  of a set $X$ is called {\em linked} if $A\cap
B\ne\emptyset$ for all $A,B\in\mathcal L$.  A linked family
$\mathcal M$ is  {\em maximal linked} if
$\mathcal M$ coincides with each linked family $\mathcal L$ on
$X$ that contains $\mathcal M$. The {\em  superextension}
$\lambda(X)$ consists of all maximal linked families on $X$. Any
associative binary operation $*: X\times X \to X$ can be extended
to an associative binary operation $*:
\lambda(X)\times\lambda(X)\to\lambda(X)$. In the
paper we study isomorphisms of superextensions of groups and prove that
two groups are isomorphic if and only if their superextensions are isomorphic.
Also we describe the automorphism groups of superextensions of all  groups
$G$ of order $|G|\leq 5$.
\end{abstract}
\maketitle

\section*{Introduction}

In this paper we investigate the automorphism group of the
superextension $\lambda(G)$ of a  group $G$. The
thorough study of various extensions of semigroups was started in
\cite{G2} and continued in \cite{BG2}--\cite{BGN}, \cite{G3}--\cite{G8}.
The largest among these extensions is the semigroup $\upsilon(S)$ of all
upfamilies on $S$. A family $\mathcal{A}$ of non-empty subsets of
a set $X$ is called an {\em upfamily} if for each set
$A\in\mathcal{A}$ any subset $B\supset A$ of $X$ belongs to
$\mathcal{A}$. Each family $\mathcal{B}$ of non-empty subsets of
$X$ generates the upfamily $\langle B \subset X :
B\in\mathcal{B}\rangle = \{ A \subset X : \exists B \in\mathcal{B}\
( B \subset A )\}$. An upfamily $\mathcal{F}$  that is closed
under taking  finite intersections is called a {\em filter}. A
filter $\mathcal{U}$ is called an {\em ultrafilter} if
$\mathcal{U} = \mathcal{F}$ for any filter $\mathcal{F}$
containing $\mathcal{U}$. The family $\beta(X)$ of all
ultrafilters on a set $X$ is called {\em the Stone-\v Cech
compactification of $X$}, see \cite{HS}, \cite{TZ}. An ultrafilter, generated by a singleton $\{ x \}$, $x\in X$, is called
{\em principal}. Each point $x\in X$ is identified with the
principal ultrafilter $\langle\{ x \}\rangle$ generated by the
singleton $\{ x \}$, and hence we can consider
$X\subset\beta(X)\subset\upsilon(X)$. It was shown in \cite{G2}
that any associative binary operation $* : S\times S \to S$ can be
extended to an associative binary operation $*:
\upsilon(S)\times\upsilon(S)\to\upsilon(S)$ defined by the formula
\begin{equation*}
\mathcal A*\mathcal B=\big\langle\bigcup_{a\in A}a*B_a:A\in\mathcal A,\;\;\{B_a\}_{a\in
A}\subset\mathcal B\big\rangle
\end{equation*}
for upfamilies $\mathcal{A}, \mathcal{B}\in\upsilon(S)$. In this
case the Stone-\v Cech compactification $\beta(S)$ is a
subsemigroup of the semigroup $\upsilon(S)$.

The semigroup $\upsilon(S)$ contains many other important
extensions of $S$. In particular, it contains the semigroup
$\lambda(S)$ of maximal linked upfamilies. The space $\lambda(S)$
is well-known in  General and Categorial Topology as the {\em
superextension} of $S$, see \cite{vM}--\cite{Ve}.   An
upfamily $\mathcal L$ of subsets  of $S$ is {\em linked} if $A\cap
B\ne\emptyset$ for all $A,B\in\mathcal L$.  The family of all
linked upfamilies on $S$ is denoted by $N_2(S)$. It is a
subsemigroup of $\upsilon(S)$. The superextension
 $\lambda(S)$ consists of all maximal elements of $N_2(S)$, see \cite{G1}, \cite{G2}.

For a finite set $X$, the cardinality of the set $\lambda(X)$ grows very quickly as $|X|$ tends to infinity.
The calculation of the cardinality of $\lambda(X)$ seems to be a
difficult combinatorial problem, which can be reformulated as the problem of counting the number
$\lambda(n)$ of  self-dual monotone Boolean functions of $n$ variables, which is well-known in Discrete Mathematics.
According to Proposition 1.1 in \cite{BMMV}, $$\log_2\lambda(n)=\frac{2^n}{\sqrt{2\pi n}}+o(1),$$ which
means that the sequence $(\lambda(n))_{n=1}^\infty$ has double exponential growth.
The sequence of numbers  $\lambda(n)$ (known in Discrete Mathematics as Ho\c sten-Morris numbers) is
included in the {\tt On-line Encyclopedia of Integer Sequences} as the sequence A001206.
All known precise values of this sequence (taken from \cite{BMMV}) are presented in the following table.
\smallskip

\begin{center}
\begin{tabular}{|r|ccccccccc|}
\hline
$|X|=$ & 1 & 2 & 3 & 4 & 5 & 6 & 7 & 8 &9\\
\hline
$|\lambda(X)|=$ & 1 & 2 & 4 & 12 & 81 & 2646 & 1422564 &229809982112&423295099074735261880\\
\hline
\end{tabular}
\end{center}
\bigskip

 Each map $f:X\to Y$ induces the map
$$\lambda f:\lambda (X)\to \lambda (Y),\quad \lambda f:\mathcal M\mapsto
\big\langle f(M)\subset Y: M\in\mathcal M\big\rangle.$$

If $\varphi: S\to S'$ is a homomorphism of semigroups, then $\lambda \varphi:
\lambda(S)\to \lambda(S')$ is a homomorphism of their superextensions, see \cite{G4}.

A non-empty subset $I$ of a semigroup $S$ is called an {\em  ideal}
if $IS\cup SI\subset I$. An ideal $I$ of a semigroup $S$ is said to be  {\em proper} if $I\neq S$.
A proper ideal $M$ of $S$ is  {\em maximal} if $M$ coincides with each proper ideal $I$ of $S$ that contains $M$.

An element $z$ of a semigroup $S$ is called {\em zero} in $S$ if $az=za=z$ for any $a\in
S$. An element $e$ of a semigroup $S$ is called an {\em idempotent} if $ee=e$. By $E(S)$ we denote the
set of all idempotents of a semigroup $S$.

Recall that an {\em isomorphism} between semigroups $S$ and $S'$ is a bijective function
$\psi:S\to S'$ such that $\psi(xy)=\psi(x)\psi(y)$ for
all $x,y\in S$. If there exists an isomorphism between $S$ and
$S'$, then $S$ and $S'$ are said to be {\em isomorphic}, denoted
$S\cong S'$. An isomorphism $\psi:S\to S$ is called an {\em automorphism} of a semigroup $S$. By $\Aut(S)$ we denote
the automorphism group of a semigroup $S$.

Following the algebraic tradition, we take for a model of a cyclic group of order $n$ the multiplicative
group $C_n=\{z\in\mathbb C:z^n=1\}$ of $n$-th roots of $1$.

For a set $X$ by $S_X$ we denote the group of all bijections of $X$.
For two sets $X\subset Y$ we shall identify $S_X$ with the subgroup $\{\varphi\in S_Y:\varphi|Y\setminus X=\mathrm{id}\}$.

\section{Extending isomorphisms from  groups to their superextensions}\label{extension}

In this section we observe that each isomorphism of  groups can be extended to an isomorphism of their superextensions
and two groups are isomorphic if and only if their superextensions are isomorphic.
The following statements are corollaries of the functoriality of the super\-extension in the category of semigroups, see~\cite{BG3b, TZ}.

\begin{proposition}\label{liso} If $\psi: G\to H$ is an isomorphism of groups, then $\lambda \psi:
\lambda(G)\to \lambda(H)$ is an isomorphism their superextensions.
\end{proposition}

\begin{corollary}\label{laut} If $\psi: G\to G$ is an automorphism of a
group $G$, then $\lambda \psi:
\lambda(G)\to \lambda(G)$ is an automorphism of the superextension $\lambda(G)$.
\end{corollary}


\begin{corollary}\label{subaut} The automorphism group $\Aut(\lambda(G))$ of the superextension of a group $G$ contains as a subgroup
an isomorphic copy of the automorphism group $\Aut(G)$ of $G$.
\end{corollary}

%
%

Corollary~\ref{subaut} motivates a question: {\em is the automorphism group $\Aut(G)$ of a group $G$  normal in the automorphism group
$\Aut(\lambda(G))$ of its superextension $\lambda(G)$?} In the next section we show that the automorphism group $\Aut(C_2\times C_2)$
of the Klein four-group $C_2\times C_2$ is not normal in $\Aut(\lambda(C_2\times C_2))$.

\begin{proposition}\label{innergr}
 Let $G$ and $H$ be  groups. If $\psi: \lambda(G)\to\lambda(H)$ is an isomorphism, then $\psi(G)=H$.
\end{proposition}

\begin{proof}
 It was shown in \cite[Proposition 1.1]{BG2} that $\lambda(G)\setminus G$ is an ideal of $\lambda(G)$. Let us prove that
$\lambda(G)\setminus G$ is the unique maximal ideal of $\lambda(G)$. Indeed, let $I$ be any ideal of $\lambda(G)$. If $g\in G\cap I$, then
$\lambda(G)=g\lambda(G)\subset I$, and hence $I=\lambda(G)$. Consequently, $\lambda(G)\setminus G$ contains each proper ideal of $\lambda(G)$.
In the same way $\lambda(H)\setminus H$ is the unique maximal ideal of $\lambda(H)$.
Taking into account that the set of maximal ideals of a semigroup is preserved by isomorphisms
and $\lambda(G)\setminus G$ and $\lambda(H)\setminus H$ are  unique maximal ideals of $\lambda(G)$ and $\lambda(H)$ respectively,
we conclude that $\psi(\lambda(G)\setminus G)=\lambda(H)\setminus H$. Therefore, $\psi(G)=H$.
\end{proof}

\begin{corollary}\label{extg} For any groups $G$ and $H$, each isomorphism from  $\lambda(G)$ to $\lambda(H)$
is an extension of an isomorphism from $G$ to $H$.
\end{corollary}

\begin{corollary}\label{ext} For any group $G$, each automorphism of  $\lambda(G)$ is an extension of an auto\-morphism of $G$.
\end{corollary}

Propositions~\ref{liso} and~\ref{innergr} imply the following theorem.

\begin{theorem}
Two groups are isomorphic if and only if their superextensions are isomorphic.
\end{theorem}

\section{The automorphism groups of the superextensions of  groups of order $\le 5$}\label{group}

In this section we shall study automorphisms of superextensions of groups and
describe the structure of the automorphism groups of superextensions $\lambda(G)$  of finite groups
 $G$ of cardinality $|G|\le5$.

Before describing the structure of extensions of finite groups,
let us make some remarks concerning the structure of a semigroup
$S$ containing a group $G$ with the identity  element which also
is a left identity  of $S$. In this case $S$ can be thought as a
$G$-space endowed with the left action of the group $G$. So we can
consider the orbit space $S/G=\{Gs:s\in S\}$ and the projection
$\pi:S\to S/G$. If $G$ lies in the center of the semigroup $S$
(which means that the elements of $G$ commute with all the
elements of $S$), then the orbit space $S/G$ admits a unique
semigroup operation turning $S/G$ into a semigroup and the orbit
projection $\pi:S\to S/G$ into a semigroup homomorphism. If $s\in
S$ is an idempotent, then the orbit $Gs$ is a group isomorphic to
a quotient group of $G$. A subsemigroup $T\subset S$ will be called a {\em
transversal semigroup} if the restriction $\pi:T\to S/G$ is an
isomorphism of the semigroups. If $S$ admits a transversal
semigroup $T$ and the elements of $G$ and $T$ commute, then $S$ is a homomorphic image of the product
$G\times T$ under the semigroup homomorphism
$$h:G\times T\to S,\;\; h:(g,t)\mapsto gt.$$
This helps to recover the algebraic structure of $S$ from the
structure of a transversal semigroup.

\smallskip

 First note that each group $G$ of cardinality $|G|\le 5$ is Abelian and is isomorphic to one of the groups: $C_1$,
$C_2$, $C_3$, $C_4$, $C_2\times C_2, C_5$.

\subsection{The semigroups $\lambda(C_1)$ and $\lambda(C_2)$} For the groups $C_n$ with
$n\in\{1,2\}$ the semigroup $\lambda(C_n)$ is isomorphic to $C_n$. Therefore, $\Aut(\lambda(C_n))\cong\Aut(C_n)\cong C_1$.

\subsection{The semigroup $\lambda(C_3)$}
For the group $C_3$ the semigroup $\lambda(C_3)$ contains three principal ultrafilters $1,z,-z$
where $z=e^{2\pi i/3}$ and the maximal linked upfamily
$\triangle=\langle\{1,z\},\{1,-z\},\{z,-z\}\rangle$ which is the zero in $\lambda(C_3)$.
The superextension $\lambda(C_3)$ is isomorphic to the multiplicative subsemigroup  $\{z\in\mathbb C:z^4=z\}$
of the complex plane.
Taking into account that the zero is preserved by  automorphisms of semigroups, we conclude that each automorphism of $C_3$ is extended to the unique
automorphism of $\lambda(C_3)$ by Corollary~\ref{ext}. Therefore, $\Aut(\lambda(C_3))\cong\Aut(C_3)\cong C_2$.

\subsection{The semigroup $\lambda(C_4)$}\label{ss:C4} Consider the cyclic group $C_4=\{1,-1,i,-i\}$.
The semigroup $\lambda(C_4)$ contains 12 elements:
$$\lambda(C_4)=\{g, g\triangle, g\Box:g\in C_4\},$$ where
$$
\triangle=\langle\{1,i\},\{1,-i\},\{i,-i\}\rangle\mbox{ \ and \ }\square=\langle\{1,i\},\{1,-i\},\{1,-1\},\{i,-i,-1\}\rangle.
$$

Taking into account that $\triangle*\Box=\Box*\triangle=\triangle$, $\triangle*\triangle=\Box*\Box=\Box$ and $C_4$
lies in the center of $\lambda(C_4)$, we conclude that
 $\lambda(C_4)$ contains a transversal semigroup $$T=\{1,\triangle,\square\},$$ where $1$
is the identity of $C_4$.

Let $\psi: \lambda(C_4)\to\lambda(C_4)$ be an automorphism. Then the restriction of $\psi$ to $C_4$ is an automorphism of $C_4$ by Proposition~\ref{innergr}.
Taking into account that $\Box$ is the unique idempotent of $\lambda(C_4)\setminus C_4$, we conclude that
$\psi(\Box)=\Box$, and hence $\psi(g\Box)=\psi(g)*\psi(\Box)=\psi(g)*\Box$ for any
$g\in C_4$. Let $\psi(\triangle)=a\triangle$ for some $a\in C_4$. Then
$$\Box=\psi(\Box)=\psi(\triangle*\triangle)=\psi(\triangle)*\psi(\triangle)=a\triangle a\triangle=a^2\Box.$$
So $a^2=1$ and $a\in\{1,-1\}$.

Since $\Aut(C_4)\cong C_2$ and each automorphism of $C_4$ can be can be extended
to an automorphism of $\lambda(C_4)$ exactly in two different ways,
the group $\Aut(\lambda(C_4))$ contains four elements.
Taking into account that each element
$\psi\in\Aut(\lambda(C_4))$ has order 2, we conclude that $\Aut(\lambda(C_4))\cong C_2\times C_2$.

\subsection{The semigroup $\lambda(C_2\times C_2)$}
The semigroup $\lambda(C_2\times C_2)$ has a similar algebraic structure. It also contains 12 elements:
$$\lambda(C_2\times C_2)=\{g, g\triangle, g\Box:g\in C_2\times C_2\},$$ where
{\small
$$
\begin{aligned}
\triangle=\;&\langle\{(1,1),(1,-1)\},\{(1,1),(-1,1)\},\{(1,-1),(-1,1)\}\rangle\mbox{ and }\\
\square=\;&\langle\{(1,1),(1,-1)\},\{(1,1),(-1,1)\},\{(1,1),(-1,-1)\},\{(1,-1),(-1,1),(-1,-1)\}\rangle.
\end{aligned}$$
}

Taking into account that $\triangle*\Box=\Box*\triangle=\triangle$, $\triangle*\triangle=\Box*\Box=\Box$ and $C_2\times C_2$
lies in the center of $\lambda(C_2\times C_2)$, we conclude that $\lambda(C_2\times C_2)$ contains a  transversal
semigroup $$T=\{e,\triangle,\square\}\subset\lambda(C_2\times C_2),$$
where $e$ is the principal ultrafilter supported by the neutral element $(1,1)$ of $C_2\times C_2$.

We shall prove that the automorphism group $\Aut(\lambda(C_2\times C_2))$ of the semigroup $\lambda(C_2\times C_2)$ is isomorphic to the holomorph $\Hol(C_2\times C_2)$ of the group $C_2\times C_2$.
\smallskip

We recall that the {\em holomorph} $\Hol(S)$ of a semigroup $S$ (see~\cite{R}) is the semi-direct product $S\rtimes \Aut(S):=(S\times \Aut(S),\star)$ of the semigroup $S$ with its automorphism group $\Aut(S)$, endowed with the semigroup operation
$$(x,f)\star (y,g)=(x\cdot f(y),f\circ g).$$
It is known\footnote{https://groupprops.subwiki.org/wiki/Holomorph\_of\_a\_group} that for the group $G=C_2\times C_2$ its holomorph $\Hol(G)$ is isomorphic to the symmetric group $S_4$.

\begin{proposition} For the group $G=C_2\times C_2$ the automorphism group $\Aut(\lambda(G))$
 is iso\-morphic to the holomorph $\Hol(G)$ of the group $G$ and hence is isomorphic to the symmetric group $S_4$.
\end{proposition}

\begin{proof}

Let $\psi: \lambda(G)\to\lambda(G)$ be an automorphism. Then the restriction of $\psi$ to $G$
is an automorphism of $G$ by Proposition~\ref{innergr}.
Taking into account that $\Box$ is the unique idempotent of $\lambda(G)\setminus G$,
we conclude that $\psi(\Box)=\Box$ and $\psi(g\Box)=\psi(g)*\psi(\Box)=\psi(g)*\Box$ for any
$g\in G$. It follows that $\psi(\triangle)=a\triangle$ for some $a\in G$.

It can be shown that for any pair $(a,f)\in G\times \Aut(G)$ the map $\psi_{(a,f)}:\lambda(C_2\times C_2)\to\lambda(C_2\times C_2)$ defined by
$$\psi_{a,f}(x)=f(x),\;\;\psi_{a,f}(x\square)=f(x)\square,\;\;\psi_{a,f}(x\triangle)=f(x)\cdot a\triangle\mbox{ \ for \ $x\in G$}$$
is an automorphism of the semigroup $\lambda(C_2\times C_2)$.

It follows that  each automorphism of $\lambda(G)$ is of the form $\psi_{a,f}$ for some $(a,f)\in G\times \Aut(G)$.

Observe that for any $(a,f),(b,g)\in G\times \Aut(G)$ and $x\in G$ we get:
$$
\begin{aligned}
&\psi_{a,f}\circ \psi_{b,g}(x)=\psi_{a,f}(g(x))=f\circ g(x),\\
&\psi_{a,f}\circ \psi_{b,g}(x\cdot \square)=\psi_{a,f}(g(x)\cdot\square)=f\circ g(x)\cdot\square,\\
&\psi_{a,f}\circ \psi_{b,g}(x\cdot \triangle)=\psi_{a,f}(g(x)\cdot b\cdot\triangle)=f\circ g(x)\cdot f(b)\cdot a\cdot\triangle.
\end{aligned}
$$
Consequently, $\psi_{a,f}\circ\psi_{b,g}=\psi_{a\cdot f(b), f\circ g}$ and hence $\Aut(\lambda(G))$
is isomorphic to the holomorph $\Hol(G)$ of the group $G$, which is known to be isomorphic to the symmetric group $S_4$.
\end{proof}

\subsection{The semigroup $\lambda(C_5)$.}\label{ss:C5}

In this subsection we describe the structure of automorphism group of the semigroup
$\lambda(C_5)$. The algebraic structure of the semigroup $\lambda(C_5)$ is described in \cite{BGN}. This semigroup contains 81 elements.
One of them is zero
$$\Z=\{A\subset C_5:|A|\ge 3\},$$ which is invariant under any bijection of $C_5$.
All the other 80 elements have 5-element orbits under the action of $C_5$,
which implies that the orbit semigroup $\lambda(C_5)/C_5$ consists of 17 elements.

It will be convenient to think of $C_5$ as the field $\{0,1,2,3,4\}$ with the
multiplicative subgroup $C_5^*=\{1,-1,2,-2\}$ of invertible elements (here $-1$
and $-2$ are identified with $4$ and $3$, respectively). Also for elements $x,y,z\in C_5$ we shall
write $xyz$ instead of $\{x,y,z\}$.

The semigroup $\lambda(C_5)$ contains 5 idempotents:
$$\begin{aligned}
\mathcal U=&\langle 0\rangle,\;\mathcal Z,\\
\Lambda_4=&\langle 01,02,03,04,1234\rangle,\\
{\Lambda}=&\langle 02,03,123,014,234\rangle,\\
2{\Lambda}=&\langle 04,01,124,023,143\rangle,
\end{aligned}$$
which commute and thus form a commutative subsemigroup $E(\lambda(C_5))$.
Being a semilattice, $E(\lambda(C_5))$ carries a natural partial order:
$e\le f$ iff $e* f=e$. The partial order $$\mathcal Z\le {\Lambda},2{\Lambda}\le \Lambda_4\le\mathcal U$$
 on the set $E(\lambda(C_5))$ is drawn in the picture:

\begin{center}
\begin{picture}(100,125)(-50,10)
\put(-4,18){$\Z$}
\put(0,30){\circle*{3}}
\put(5,35){\line(1,1){20}}
\put(-5,35){\line(-1,1){20}}
\put(-30,60){\circle*{3}}
\put(-45,55){${\Lambda}$}

\put(30,60){\circle*{3}}
\put(35,55){$2{\Lambda}$}
\put(-4,78){$\Lambda_4$}
\put(0,90){\circle*{3}}
\put(0,95){\line(0,1){20}}
\put(-5,35){\line(-1,1){20}}
\put(5,85){\line(1,-1){20}}
\put(-5,85){\line(-1,-1){20}}
\put(0,120){\circle*{3}}
\put(-3,125){$\mathcal U$}
\end{picture}
\end{center}

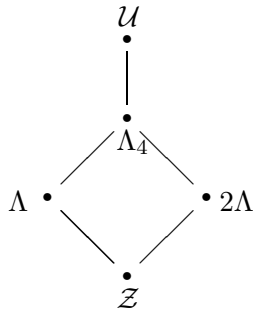
\centerline{Diagram 1. The structure of the semilattice $E(\lambda(C_5))$}
\medskip

Next, consider two subsets:
$$\sqrt{\mathcal Z}=\{\mathcal L\in\lambda(C_5):\mathcal L*\mathcal L=\mathcal Z\}$$ and
$$
\sqrt{E(\lambda(C_5))}=\{\mathcal L\in\lambda(C_5):\mathcal L*\mathcal L\in E(\lambda(C_5))\}=\{\mathcal L\in\lambda(C_5):\mathcal L*\mathcal L*\mathcal L*\mathcal L=\mathcal L*\mathcal L\}.
$$ We claim that the set $\sqrt{E(\lambda(C_5))}\setminus\sqrt{\mathcal Z}$ has at most one-point intersection with each orbit.
 Indeed, if $\mathcal L\in\sqrt{E(\lambda(C_5))}$ and $\mathcal L*\mathcal L\ne\Z$,
 then  for every $a\in C_5\setminus\{0\}$, we get
$$
\begin{aligned}
(\mathcal L+a)*(\mathcal L+a)*(\mathcal L+a)*(\mathcal L+a)=&\,\mathcal L*\mathcal L*\mathcal L*\mathcal L+4a=\\
=&\,\mathcal L*\mathcal L+4a\ne \mathcal L*\mathcal L+2a=(\mathcal L+a)*(\mathcal L+a).
\end{aligned}
$$witnessing that $\mathcal L+a\notin\sqrt{E(\lambda(C_5))}$.

By a direct calculation one can check that the set $\sqrt{E(\lambda(C_5))}$  contains the
following four maximal linked upfamilies:$$\begin{aligned}
\Delta=&\langle 02,03,23\rangle,\\
\Lambda_3=&\langle 02,03,04,234\rangle,\\
\Theta=&\langle 14,012,013,123,024,034,234\rangle,\\
\Gamma=&\langle 02,04,013,124,234\rangle.
\end{aligned}
$$
For those upfamilies we get
$$
\begin{aligned}
&\Delta* \Delta=\Delta*\Delta*\Delta={\Lambda},\\
&\Lambda_3*\Lambda_3=\Lambda_3*\Lambda_3*\Lambda_3={\Lambda},\\
&\F*\Theta=\F*\Gamma=\Z\;\; \mbox{for every $\F\in\lambda(C_5)\setminus C_5.$}
\end{aligned}
$$
All the other elements of $\lambda(C_5)$ can be found as  images of
$\Delta,\Theta,\Gamma,\Lambda_3$ under the affine transformations of the
 field $C_5$. Those are maps of the form $$f_{a,b}:x\mapsto ax+b \mod 5,$$
 where $a\in\{1,-1,2,-2\}=C_5^*$ and $b\in C_5$. The image of a maximal linked
  upfamily $\mathcal L\in\lambda(C_5)$ under such a transformation will be denoted by $a\mathcal L+b$.

One can check that $a\Lambda_4=\Lambda_4$ for each $a\in C_5^*$
while ${\Lambda}=-{\Lambda}$, and $\Theta=-\Theta$. Since the
linear transformations of the form $f_{a,0}:C_5\to C_5$, $a\in
C_5^*$, are authomorphisms of the group $C_5$, the induced
transformations $\lambda f_{a,0}:\lambda(C_5)\to\lambda(C_5)$
are authomorphisms of the semigroup $\lambda(C_5)$. This implies
that those transformations do not move the subsets
$E(\lambda(C_5))$ and $\sqrt{E(\lambda(C_5))}$. Consequently,
the set $\sqrt{E(\lambda(C_5)}$ contains the maximal linked
upfamilies:
$$a\Delta,a\Theta,a\Lambda_3,a\Gamma, \;a\in\IZ^*_5,$$ which together with the idempotents form a 17-element subset
$$T_{17}=E(\lambda(C_5))\cup\big\{a\Delta,a\Theta:a\in\{1,2\}\big\}\cup \{a\Lambda_3,a\Gamma:a\in\IZ^*_5\}$$
that projects bijectively onto the orbit semigroup $\lambda(C_5)/C_5$. The set $T_{17}$ looks as follows
 (we connect an element $x\in T_{17}$ with an idempotent $e\in T_{17}$ by an arrow if $x* x=e$):

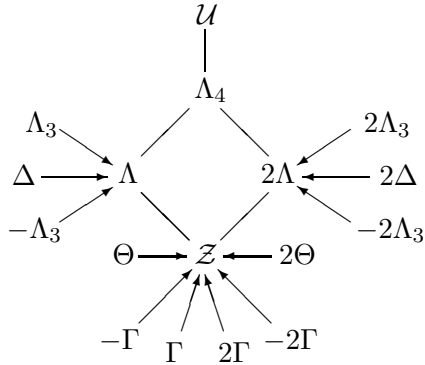
\begin{center}
\begin{picture}(100,170)(-50,-20)
\put(-4,27){$\Z$}
\put(6,36){\line(1,1){18}}
\put(-6,36){\line(-1,1){18}}
\put(-33,58){$\Lambda$}

\put(21,58){$2\Lambda$}
\put(-4,90){$\Lambda_4$}
\put(0,100){\line(0,1){16}}
\put(-6,36){\line(-1,1){18}}
\put(6,86){\line(1,-1){18}}
\put(-6,86){\line(-1,-1){18}}
\put(-3,118){$\mathcal U$}

\put(-40,-5){$-\Gamma$}
\put(-35,27){$\Theta$}
\put(-25,30){\vector(1,0){18}}
\put(28,27){$2\Theta$}

\put(25,30){\vector(-1,0){18}}
\put(-25,5){\vector(1,1){20}}
\put(-15,-10){$\Gamma$}
\put(-9,2){\vector(1,3){7}}
\put(5,-10){$2\Gamma$}
\put(7,2){\vector(-1,3){7}}
\put(22,-5){$-2\Gamma$}
\put(25,5){\vector(-1,1){20}}

\put(-75,37){$-\Lambda_3$}
\put(-55,43){\vector(3,2){20}}
\put(-73,57){$\Delta$}
\put(-62,60){\vector(1,0){25}}
\put(-68,77){$\Lambda_3$}
\put(-55,78){\vector(3,-2){20}}

\put(57,37){$-2\Lambda_3$}
\put(55,43){\vector(-3,2){20}}
\put(66,57){$2\Delta$}
\put(62,60){\vector(-1,0){25}}
\put(60,77){$2\Lambda_3$}
\put(55,78){\vector(-3,-2){20}}
\end{picture}
\end{center}

It follows that $\sqrt{E(\lambda(C_5))}=T_{17}\cup\sqrt{\Z}$ where $\sqrt{\Z}=\{\Theta,2\Theta,\Gamma,2\Gamma,-\Gamma,-2\Gamma\}+C_5$.

Since each element of $\lambda(C_5)$ can be uniquely written as the sum $\mathcal L+b$ for
 some $\mathcal L\in T_{17}$ and  $b\in C_5$, the multiplication table for the semigroup
  $\lambda(C_5)$ can be recovered from the following Cayley table 
   for multiplication of the elements of the set $T_{17}$.
\smallskip

\begin{center}
\resizebox{16.5cm}{!}{
\begin{tabular}{||c||c|cccc|cccc|c||}
\hhline{|t:=:t:==========:t|}
$*$ & $\Lambda_4$ & ${\Lambda}$ & $\Delta$ & $\Lambda_3$ & $-\Lambda_3$ & $2{\Lambda}$ & $2\Delta$ & $2\Lambda_3$ & $-2\Lambda_3$ & $a\Theta,a\Gamma$\\
\hhline{|:=::==========:|}

$\Lambda_4$ & $\Lambda_4$ & ${\Lambda}$ & ${\Lambda}$  & ${\Lambda}$ & ${\Lambda}$  & $2{\Lambda}$ & $2{\Lambda}$ & $2{\Lambda}$ & $2{\Lambda}$ & $\Z$\\
\hhline{||-||----------||}

${\Lambda}$ & ${\Lambda}$ & ${\Lambda}$ & ${\Lambda}$  & ${\Lambda}$ & ${\Lambda}$  & $\Z$ & $\Z$ & $\Z$ & $\Z$ & $\Z$\\
\hhline{||-||----------||}
$\Delta$ & $\Delta$ & ${\Lambda}$ & ${\Lambda}$  & ${\Lambda}$ & ${\Lambda}$  & $2\Theta$ & $2\Theta$ & $2\Theta$ & $2\Theta$ & $\Z$\\

$\Lambda_3$ & $\Lambda_3$ & ${\Lambda}$ & ${\Lambda}$  & ${\Lambda}$ & ${\Lambda}$  & $2\Theta+2$ & $2\Theta+2$ & $2\Theta+2$ & $2\Theta+2$ & $\Z$\\

$-\Lambda_3$ & $-\Lambda_3$ & ${\Lambda}$ & ${\Lambda}$ & ${\Lambda}$ & ${\Lambda}$  & $2\Theta-2$ & $2\Theta-2$ & $2\Theta-2$ & $2\Theta-2$ & $\Z$\\
\hhline{||-||----------||}

$2{\Lambda}$ & $2{\Lambda}$ & $\Z$ & $\Z$ & $\Z$ & $\Z$ & $2{\Lambda}$ & $2{\Lambda}$  & $2{\Lambda}$ & $2{\Lambda}$ & $\Z$\\
\hhline{||-||----------||}

$2\Delta$ & $2\Delta$ & $\Theta$ & $\Theta$ & $\Theta$ & $\Theta$ & $2{\Lambda}$ & $2{\Lambda}$  & $2{\Lambda}$ & $2{\Lambda}$ & $\Z$\\

$2\Lambda_3$ & $2\Lambda_3$ & $\Theta-1$ & $\Theta-1$ & $\Theta-1$ & $\Theta-1$ & $2{\Lambda}$ & $2{\Lambda}$  & $2{\Lambda}$ & $2{\Lambda}$ & $\Z$\\

$-2\Lambda_3$ & $-2\Lambda_3$ & $\Theta+1$ & $\Theta+1$ & $\Theta+1$ & $\Theta+1$ & $2{\Lambda}$ & $2{\Lambda}$  & $2{\Lambda}$ & $2{\Lambda}$ & $\Z$\\
\hhline{||-||----------||}

$\Theta$ & $\Theta$ & $\Theta$ & $\Theta$ & $\Theta$ & $\Theta$ & $\Z$ & $\Z$ & $\Z$ & $\Z$ & $\Z$\\

$2\Theta$ & $2\Theta$ & $\Z$ & $\Z$ & $\Z$ & $\Z$ & $2\Theta$ & $2\Theta$ & $2\Theta$ & $2\Theta$ & $\Z$\\
\hhline{||-||----------||}

$\Gamma$ & $\Gamma$ & $\Theta+1$ & $\Theta+1$ & $\Theta+1$ & $\Theta+1$ & $2\Theta+2$ & $2\Theta+2$ & $2\Theta+2$ & $2\Theta+2$ & $\Z$\\

$-\Gamma$ & $-\Gamma$ & $\Theta-1$ & $\Theta-1$ & $\Theta-1$ & $\Theta-1$ & $2\Theta-2$ & $2\Theta-2$ & $2\Theta-2$ & $2\Theta-2$ & $\Z$\\

$2\Gamma$ & $2\Gamma$ & $\Theta-1$ & $\Theta-1$ & $\Theta-1$ & $\Theta-1$ & $2\Theta+2$ & $2\Theta+2$ & $2\Theta+2$ & $2\Theta+2$ & $\Z$\\

$-2\Gamma$ & $-2\Gamma$ & $\Theta+1$ & $\Theta+1$ & $\Theta+1$ & $\Theta+1$ & $2\Theta-2$ & $2\Theta-2$ & $2\Theta-2$ & $2\Theta-2$ & $\Z$\\

\hhline{|b:=:b:==========:b|}
\end{tabular}
}
\end{center}


Now we are able to prove the main result of this subsection.

\begin{theorem}\label{t:C5} $\Aut(\lambda(C_5))\cong \Aut(C_5)\cong C_4$.
\end{theorem}

\begin{proof} We identify the group $\Aut(C_5)$ with the subgroup $\{\lambda\varphi:\varphi\in\Aut(C_5)\}$ of $\Aut(\lambda(C_5))$. To see that $\Aut(\lambda(C_5))=\Aut(C_5)$, it suffices to prove that an automorphism $\psi$ of $\lambda(C_5)$ is identity if its restriction $\psi|C_5$ is the identity automorphism of the group $C_5$. So, we assume that $\psi(x)=x$ for all $x\in C_5$. Since each element $\A\in\lambda(C_5)$ is of the form $\A=\mathcal L+x$ for some $\mathcal L\in T_{17}$ and $x\in C_5$, it suffices to prove that $\psi(\mathcal L)=\mathcal L$ for every $\mathcal L\in T_{17}$.

Let us call a subset $A\subset \lambda(C_5)$ \ {\em $\psi$-invariant} if $\psi(A)=A$. An element $a\in\lambda(C_5)$ is defined to be $\psi$-{\em invariant} if the singleton $\{a\}$ is $\psi$-invariant. Observe that for any $\psi$-invariant sets $A,B\subset\lambda(C_5)$ the set $A* B=\{a* b:a\in A,\;b\in B\}$ is $\psi$-invariant. In particular, the set $A+x$ is $\psi$-invariant for every $x\in C_5$.

Since $\psi$ is an automorphism of the semigroup $\lambda(C_5)$, the set $E(\lambda(C_5))=\{\mathcal U,\mathcal Z,\Lambda_4,\Lambda,2\Lambda\}$ of idempotents of $\lambda(C_5)$ is $\psi$-invariant and $\psi|E(\lambda(C_5))$ is an automorphism of the semilattice $E(\lambda(C_5))$. Looking at the structure of the semilattice
$E(\lambda(C_5))$, we can see that its elements $\mathcal Z,\mathcal U$ and $\Lambda_4$ are $\psi$-invariant and hence $\psi(\Lambda)\in\{\Lambda,2\Lambda\}$.

Two cases are possible:

1) $\psi(\Lambda)=\Lambda$. In this case the following sets are $\psi$-invariant: $$\begin{gathered}
\{\Lambda\},\;\;\{2\Lambda\},\\
\sqrt{\Lambda}\setminus\{\Lambda\}=\{\Delta,\Lambda_3,-\Lambda_3\},\;\; \sqrt{\{2\Lambda\}}\setminus\{2\Lambda\}=\{2\Delta,-2\Lambda_2,2\Lambda_3\},\\
\{2\Delta,-2\Lambda_2,2\Lambda_3\}* \{\Lambda\}=\{\Theta,\Theta-1,\Theta+1\},\;\;  \{\Delta,\Lambda_3,-\Lambda_3\}* \{2\Lambda\}=\{2\Theta,2\Theta+2,2\Theta-2\}.
\end{gathered}
$$

It follows that the set $\{\Theta,\Theta-1,\Theta+1\}+2=\{\Theta+2,\Theta+1,\Theta+3\}$ is $\psi$-invariant and so are the sets  $\{\Theta+1\}=\{\Theta,\Theta-1,\Theta+1\}\cap\{\Theta+2,\Theta+1,\Theta+3\}$ and $\{\Theta\}$. By analogy we can prove that the element $2\Theta$ is $\psi$-invariant.

Since the elements $\Lambda$, $2\Lambda$, $\Theta$ and $2\Theta$ are $\psi$-invariant, the set $$\{\mathcal L\in\lambda(C_5):\mathcal L*\mathcal L=\Lambda,\;\mathcal L* \Lambda=\Lambda,\;\mathcal L* 2\Lambda=2\Theta\}=\{\Delta\}$$ is $\psi$-invariant and hence the element $\Delta$ is $\psi$-invariant. By analogy we can prove that the element $2\Delta$, $\Lambda_3$, $2\Lambda_3$, $-\Lambda_3$, $-2\Lambda_3$ are $\psi$-invariant.

Since the elements $\Z$, $\Lambda$, $\Theta+1$, $2\Theta+2$ are $\psi$-invariant, the set
$$\{\mathcal L\in\lambda(C_5):\mathcal L*\mathcal L=\mathcal Z,\;\mathcal L* \Lambda=\Theta+1,\;\mathcal L* 2\Lambda=2\Theta+2\}=\{\Gamma\}$$
is $\psi$-invariant. So, the element $\Gamma$ is $\psi$-invariant.
By analogy we can prove that the elements $-\Gamma$, $2\Gamma$, $-2\Gamma$ are $\psi$-invariant.

Therefore, $\psi(\mathcal L)=\mathcal L$ for all $\mathcal L\in T_{17}$ and $\psi$ is the identity automorphism of $\lambda(C_5)$.
\smallskip

2) $\psi(\Lambda)\ne \Lambda$. Then $\psi(\Lambda)=2\Lambda$, $\psi(2\Lambda)=\Lambda$, and by the first case, $\psi\circ\psi$ is the identity automorphism of $\lambda(C_5)$. It follows that $$\psi(\{\Delta,\Lambda_3,-\Lambda_3\})=\psi(\sqrt{\Lambda}\setminus\{\Lambda\})=
\sqrt{2\Lambda}\setminus\{2\Lambda\}=\{2\Delta,2\Lambda_3,-2\Lambda_3\}$$
and
$$\psi(\{\Theta,\Theta{-}1,\Theta{+}1\}){=}\psi(\{2\Delta,2\Lambda_3,{-}2\Lambda_3\}{*} \{\Lambda\})=\{\Delta,\Lambda_3,{-}\Lambda_3\}{*}\{2\Lambda\}=
\{2\Theta,2\Theta{+}2,2\Theta{-}2\}.$$ Then
$\psi(\{\Theta+1,\Theta,\Theta+2\})=
\psi(\{\Theta,\Theta-1,\Theta+1\}+1)=\{2\Theta,2\Theta+2,2\Theta-2\}+1=\{2\Theta+1,2\Theta-2,2\Theta-1\}$ and hence $\psi(\{\Theta,\Theta+1\})=\psi(\{\Theta,\Theta-1,\Theta+1\}
\cap\{\Theta+1,\Theta,\Theta+2\})=
\{2\Theta,2\Theta+2,2\Theta-2\}\cap\{2\Theta+1,2\Theta-2,2\Theta-1\}=\{2\Theta-2\}$, which contradicts the bijectivity of $\psi$. So, the case $\psi(\Lambda)\ne\Lambda$ is impossible and by the first case, $\psi$ is the identity automorphism of $\lambda(C_5)$.

Therefore, $\Aut(\lambda(C_5))=\Aut(C_5)\cong C_4$.
\end{proof}

\subsection{The final table}

We summarize the obtained results on the automorphism groups $\Aut(\lambda(G))$
of super\-exten\-sions of groups $G$ of cardinality $|G|\leq 5$ in the following table:
\smallskip

\begin{center}
\begin{tabular}{|c|cccccc|}
\hline
$G$ & $C_1$ & $C_2$ & $C_3$ & $C_4$ & $C_2\times C_2$&$C_5$\\
\hline
$\Aut(G)$ & $C_1$& $C_1$ & $C_2$ & $C_2$ & $S_3$ &$C_4$\\
\hline
$\Aut(\lambda(G))$ & $C_1$ & $C_1$ & $C_2$ & $C_2\times C_2$ & $S_4$ &$C_4$\\
\hline
\end{tabular}
\end{center}
\medskip

Analyzing the entries of this table, we can ask the following question.

\begin{problem} Is  $\Aut(\lambda(G))$ isomorphic to $\Aut(G)$ for a finite (cyclic) group $G$ of odd order?
\end{problem}


\end{document}